\documentclass[smallextended]{svjour3}
\journalname{Graphs and Combinatorics}
\def\makeheadbox{}
\def\emailname{e-mail}
\usepackage{amsfonts}
\usepackage{amssymb}
\usepackage{latexsym}
\usepackage{tikz}
\begin{document}
\newtheorem{thm}{Theorem}
\newtheorem{cor}[thm]{Corollary}
\newcommand\Prob{\mathbf{Pr}}
\newcommand\bc{\mathsf{bc}}
\newcommand\bp{\mathsf{bp}}
\newcommand\BP{\mathsf{CB}}
\newcommand\MP{\mathsf{CM}}
\newcommand\vc{\mathsf{vc}}
\newcommand\vp{\mathsf{vp}}
\newcommand\vcf{\mathsf{vc}^*}
\newcommand\vpf{\mathsf{vp}^*}
\newcommand\HH{\mathcal H}
\newcommand\load{\mathsf{load}}
\newcommand\F{\mathcal F}

\title{Erd\H os-Pyber theorem for hypergraphs and secret sharing}
\author{L\'aszl\'o Csirmaz \and P\'eter Ligeti \and G\'abor Tardos}
\institute{L. Csirmaz\\
Central European University, Budapest, Hungary\\ 
\email{csirmaz@renyi.hu} \\[5pt]
P. Ligeti\\
Department of Computeralgebra and ELTECRYPT Research Group, \\
E\"otv\"os Lor\'and University,
Budapest, Hungary \\
\email{turul@cs.elte.hu}\\[5pt]
G. Tardos \\
R\'enyi Institute of Mathematics,
Budapest, Hungary\\
\email{tardos.gabor@renyi.mta.hu}}
\date{}
\maketitle

\begin{abstract}
A new, constructive proof with a small explicit constant is given to the
Erd\H os-Pyber theorem which says that the edges of a graph on $n$ vertices
can be partitioned into complete bipartite subgraphs so that every vertex is
covered at most $O(n/\log n)$ times. 
The theorem is generalized to uniform
hypergraphs. Similar bounds with smaller constant value is provided for
fractional partitioning both for graphs and for uniform hypergraphs. We show
that these latter constants cannot be improved by more than a factor of 1.89
even for fractional covering by arbitrary complete multipartite subgraphs or
subhypergraphs.
In the case every vertex of the graph is connected to at least $n-m$ other
vertices, we prove the existence of a fractional covering of the edges by 
complete bipartite graphs such that every vertex is covered at most
$O(m/\log m)$ times, with only a slightly worse explicit constant. This
result also generalizes to uniform hypergraphs.
Our results give new improved bounds on the complexity of graph and
uniform hypergraph based secret sharing schemes, and show the limits of the
method at the same time.

\medskip
\end{abstract}

\section{Introduction}\label{sec:intro}

While graph decomposition is an interesting topic by itself
\cite{alon,bezrukov,fishburn,hajuabolhassan,jukna,katona-szemeredi,pinto},
our interest comes mainly from a particular application in cryptography. Upper
bounds on the worst case complexity of secret sharing schemes often use
graph decomposition techniques \cite{beimel-etal,padro,kn:stinson}. In
this respect the most cited result is the Erd\"os-Pyber theorem
\cite{erdospyber}, which gives the estimate $O(n/\log n)$ on the vertex
cover number of a graph $G$ on $n$ vertices when the edges of $G$ are
partitioned into complete bipartite subgraphs of $G$. (In this paper
$\log$ denotes the base $2$ logarithm.) An immediate consequence is that for
any graph there is secret sharing scheme realizing that graph with
complexity $O(n/\log n)$ -- the best upper bound known for the share size in
general graphs.  Any improvement in the the Erd\H os-Pyber theorem implies
immediately a similar improvement in this bound. In this paper we give an alternate 
proof of the Erd\H os-Pyber 
theorem which yields an almost optimal explicit constant, and prove a
generalization for $d$-uniform hypergraphs, again with an almost optimal
explicit constant term.

Motivated by the question of what makes a graph ``hard'' for secret sharing
schemes, Beimel et al \cite{beimel-etal} considered very dense graphs, that
is, graphs where every node has high degree. In this paper we look at the
cover number of graphs on $n$ vertices where every vertex has degree at
least $n-m$. Our result implies that all of these graphs can be realized 
by secret sharing schemes with complexity $O(m/\log m)$ with an explicit
constant term. Setting $m=n$ we get yet another proof of the Erd\"os-Pyber
theorem, in this case the constant is only 1.443 times larger than the 
constant in our first proof.

Finally, in the other direction we show that our results are optimal within
a factor of 1.89. In particular, we show that for large enough $n$ there are 
graphs
($d$-uniform hypergraphs) on $n$ vertices where the fractional cover number
is more than 0.53 times our upper bound, even if arbitrary complete
multipartite graphs are allowed in the covering. This result implies that
the technique of covering $G$ with complete multipartite graphs cannot
improve the known upper bound of $O(n/\log n)$ on the complexity of $G$.

The paper is organized as follows. In Sect.~\ref{sec:results} we recall
some definitions and notations, and present our results. Our constructive
proof of the Erd\H os-Pyber theorem is given in Sect.~\ref{sec:graph}.
The generalization for $d$-uniform hypergraphs is proved in 
Sect.~\ref{sec:hypergraph}, Sect.\ref{sec:dense} deals with the case of
dense
graphs. Finally in Sect.~\ref{sec:this-is-best} we prove that our results
are the best possible up to a small constant multiplier.

\section{Preliminaries and results}\label{sec:results}

\subsection{Graph decomposition}

Let $\F$ be a collection of graphs. An $\F$-graph is a graph which is
isomorphic to an element of $\F$. An {\em $\F$-cover} of a graph $G=(V,E)$
is a collection of subgraphs of $G$ that are $\F$-graphs, and the union of 
whose
edge sets is $E$. An $\F$-cover is an {\em $\F$-partition} if the edge sets of
the subgraphs are pairwise disjoint. Given an $\F$-cover of $G$, the {\em
  load} of a vertex
$v\in V$ is the number of subgraphs containing it. The {\em vertex $\F$-cover
number} and the {\em vertex $\F$-partition number} of a graph $G$ is the
smallest $r$ such that for some $\F$-cover (or $\F$-partition, respectively)
of $G$ all vertex loads are at most $r$. We denote these numbers by
$\vc_{\F}(G)$, and
$\vp_{\F}(G)$, respectively, and the value is $+\infty$ when no such a cover
or partition exists. For more information on this notation and
a discussion of its relevance, see, e.g., \cite{knauer-ueckerdt}.

In the {\em fractional}, or weighted version, each $\F$-subgraph of a graph
$G=(V,E)$ has a
non-negative weight, and the total weight of the subgraphs containing an edge
$e\in E$ should be at
least 1 (cover), or exactly 1 (partition). In this case the load of a
vertex is the sum of the weights of the $\F$-subgraphs containing it, and the
corresponding {\em fractional vertex $\F$-cover number} and {\em fractional
vertex $\F$-partition number} are denoted by $\vcf_\F(G)$, and $\vpf_\F(G)$,
respectively. It is clear that
\begin{equation}\label{eq:basicineq}
   \vcf_\F(G) \le \vc_\F(G) \le \vp_\F(G),~\mbox{ and }~ 
   \vcf_\F(G)\le \vpf_\F(G) \le \vp_\F(G).
\end{equation}
The most frequently investigated case is when $\F$ is the collection of
complete bipartite graphs, denoted here by $\BP$. In the classical work of Fishburn and Hammer
\cite{fishburn} $\vc_\BP(G)$ is called the bipartite degree of $G$. Dong and Liu
showed in \cite{dong-liu} that $\vc_\BP(K_n)=\vp_\BP(K_n)=\lceil\log 
n\rceil$,\footnote{Note that $\log$ denotes base $2$ logarithm.} and that
$\vp_\BP(G)\le 4$ for planar graphs. Pinto \cite{pinto} calls $\vc_\BP(G)$ and
$\vp_\BP(G)$ the local biclique cover and partition number of $G$, respectively,  and shows
that there are graphs with $\vc_\BP(G)=2$ while $\vp_\BP(G)$ can be arbitrary
large. V.~Watts investigates fractional partitions and covers in \cite{watts}.

Other well studied graph families are the collection of stars, and the 
collection of cycles. We will use another important graph family, that of 
the {\em complete multipartite graphs}. These graphs are the complements of
disjoint unions of complete graphs and we denote their collection by $\MP$.

\subsection{$d$-uniform hypergraphs}

Let $d\ge 2$ be an integer. A {\em $d$-uniform hypergraph} $\HH$ is a pair
$(V,E)$, where $V$ is the set of vertices, and $E$ is the set of edges,
often called {\em hyperedges}, and each edge is a $d$-element subset of $V$.
A {\em   subhypergraph} of $\HH$ is a $d$-uniform hypergraph 
$(V',E')$ with $V'\subseteq V$ and $E'\subseteq E$.

A $d$-uniform hypergraph $\HH$ is a {\em complete $d$-uniform $k$-partite
hypergraph}, a {\em $(d,k)$-cuph} for short, if its vertex set can be
partitioned into $k$ parts such that the edge set consists of the subsets of
the vertices that intersect $d$ of the parts, each in a single vertex. We
call the parts of this partition the {\em partite sets} of $\cal H$.
When $k=d$ we call a $(d,d)$-cuph simply a $d$-cuph.

With a slight abuse of notation, the family of complete $d$-uniform multipartite
hypergraphs ($(d,k)$-cuphs) is also denoted by $\MP$, and the family of 
$d$-cuphs is denoted by $\BP$. In these definitions we consider $d$ to be
fixed but the class $\MP$ contain $(d,k)$-cuphs for arbitrary $k$.
Note that in the $d=2$ case we get back the the standard notion of 
complete multipartite and bipartite graphs, respectively.

\medskip

When $\F$ is a collection of $d$-uniform hypergraphs, the notion of
$\F$-cover and $\F$-partition as well as the load of a vertex generalizes
easily for $d$-uniform hypergraphs. The values $\vc_\F(\HH)$, 
$\vp_\F(\HH)$ and their fractional versions are defined similarly as has
been done for standard graphs. Inequalities in (\ref{eq:basicineq})
remain valid in this case.

\subsection{Secret sharing}

A secret sharing scheme, introduced in \cite{blakley,shamir} is a
probabilistic method by
which a dealer, who holds a secret, distributes shares to a set of
participants, so that only {\em authorized} subsets of the participants are
able to reconstruct the secret from their shares. The collection of all
authorized subsets is called the {\em access structure}. We only consider
{\em perfect schemes}, in which unauthorized subsets of participants should learn
nothing about the secret, that is, the collection of their shares should be
independent of the secret. Secret sharing schemes are considered as one of the
main building blocks in modern cryptography \cite{padro}. Most research on
secret sharing focuses on the ratio between the size of the largest share
and the size of the secret. Size is measured here by way of entropy. The
{\em complexity}, or information ratio of
an access structure is the infimum of this ratio over all schemes realizing
the structure. In this paper we consider access structures where all
minimal authorized subsets have the same size $d\ge 2$.  These access
structures can be described by $d$-uniform hypergraphs, where each vertex
represents a participant, and $d$ vertices form a hyperedge if
and only if the respective $d$-element set of participants is authorized. For such a hypergraph $\HH$,
the complexity of the access structure based on $\HH$ is denoted by
$\sigma(\HH)$. Hypergraphs with at least one edge have complexity
at least 1. Hypergraphs with complexity exactly 1 are called {\em ideal}. 
The complete $d$-uniform multipartite hypergraphs, that is  $(d,k)$-cuphs, 
are ideal
\cite{padro}. When $d=2$ all other non-trivial graphs have complexity at
least $3/2$ \cite{blundo-etal}, for $d\ge 3$ the characterization of ideal
$d$-uniform hypergraphs is an open problem.

Our interest in graph decomposition stems from Stinson's Decomposition
Theorem \cite{kn:stinson} which is an indispensable tool in giving
upper bounds on the complexity of access structures. While Stinson's 
theorem is more general, we state here in a special case.
\begin{thm}[Stinson \cite{kn:stinson}]\label{thm:stinson}
Let $\F$ be any collection of ideal $d$-uniform hypergraphs.
For any $d$-uniform hypergraph $\HH$ we have $\sigma(\HH) \le \vcf_\F(\HH)$. \qed
\end{thm}
As complete $d$-uniform multipartite graphs are ideal, the complexity of
the hypergraph $\HH$ can be upper bounded by $\vp_\MP(\HH)$.

\subsection{Our results}

In the asymptotic notation 
$O(\cdot)$ and $o(\cdot)$ we assume that $d$ is fixed and $n$ and in the case
of Theorem~\ref{uj} and Corollary~\ref{cor:complexity}, also $m$ tend to
infinity.

Erd\H os and Pyber proved in \cite{erdospyber} that for any (standard)
graph $G$ on $n$ vertices, the edge set of $G$ can be partitioned into
complete bipartite graphs so that every vertex of $G$ is contained in at most
$O(n/\log n)$ of the bipartite graphs. Using the notation introduced
above, their result
can be expressed equivalently as  $\vp_\BP(G)=O(n/\log n)$.
They also remarked that
this estimate is the best possible. We give a constructive proof of 
the Erd\H os-Pyber theorem with an improved explicit constant factor. 
\begin{thm}\label{thm:improved-EP}
If $G$ is a graph on $n$ vertices, then
$$
    \vp_\BP(G) \le \big(1+o(1)\big)\frac{n}{\log n},
$$
moreover 
$$
    \vpf_\BP(G) \le \big(0.5+o(1)\big)\frac{n}{\log n} .
$$
\end{thm}

We prove a generalization of this theorem to $d$-uniform 
hypergraphs with higher values of $d$ as follows.

\begin{thm}\label{thm:d-uniform}
Let $d\ge 2$ be an integer, and $\HH$ be a $d$-uniform hypergraph on
$n$ vertices. In this case
$$
\vp_\BP(\HH) \le \big(\frac{1}{(d-2)!} +o(1)\big)\frac{n^{d-1}}{\log n} ,
$$
and
$$\vpf_\BP(\HH) \le \big(\frac{1}{d!}+o(1)\big)\frac{n^{d-1}}{\log n} .
$$
\end{thm}
\medskip

In case a graph $G$ is dense and each vertex is connected to almost all other
vertices, the bound on $\vcf_\BP(G)$ implied by Theorem~\ref{thm:improved-EP} can
be strengthened. Similar strengthening works for dense
hypergraphs. Note that we
can choose $m=n$ to get a result for arbitrary graphs or hypergraphs, 
and even in this case the bounds are only slightly worse than the ones
implied by Theorems~\ref{thm:improved-EP} and \ref{thm:d-uniform}.

\begin{thm}\label{uj}
If $G$ is a graph on $n$ vertices such that every vertex has
degree at least $n-m$, then we have
$$\vcf_\BP(G) \le(0.722+o(1))\frac m{\log m}.$$
If the $d$-uniform hypergraph $\HH$ on $n$ vertices satisfies that every set
of $d-1$ vertices that appears together in an edge appears in at least $n-m$
edges, then
$$\vcf_\BP(\HH)\le\left(\frac{1.443}{d!}+o(1)\right)\frac{n^{d-2}m}{\log m}.$$
\end{thm}
\medskip

Our results can be applied to get universal bounds on the complexity of graph
and uniform hypergraph based structures.

\begin{cor}\label{cor:complexity}
For any graph $G$ on $n$ vertices, $\displaystyle\sigma(G)\le (1/2+o(1))\frac{n}{\log n}$.
If the graph $G$ has minimum degree $n-m$, then
$\displaystyle\sigma(G)\le(0.722+o(1))\frac{m}{\log m}$.

For any d-uniform hypergraph $\HH$ on $n$ vertices, $\displaystyle\sigma(\HH)\le
(\frac{1}{d!}+o(1))\frac{n^{d-1}}{\log n}$.
If every set of $d-1$ vertices in $\HH$ that appears in a hyperedge appears in
at least $n-m$ of them, then we also have
$\displaystyle\sigma(\HH)\le(\frac{1.443}{d!}+o(1))\frac{n^{d-2}m}{\log m}$.
\end{cor}
\medskip

From the other direction we show that the fractional
results in Theorems~\ref{thm:improved-EP} and \ref{thm:d-uniform} cannot be
improved by more than a factor of $1.89$
even if we consider fractional covering instead of fractional partition and
arbitrary complete multipartite hypergraphs instead of $d$-cuphs.

\begin{thm}\label{thm:upper-bound}
For every $d\ge 2$ and $n\ge n_0(d)$ there is a
$d$-uniform hypergraph $\HH$ on $n$ vertices such that
$$ \vcf_\MP(\HH) \ge \frac{0.53}{d!}\cdot\frac{n^{d-1}}{\log n}.
$$
\end{thm}
This theorem indicates that new and different ideas are required to
improve the general upper bound on the
complexity of graphs and hypergraphs given by Corollary \ref{cor:complexity}.

\section{Graphs - proof of the Erd\H os-Pyber theorem}\label{sec:graph}

With the choice $k=\lceil\log n-2\log\log n\rceil$ the following lemma implies
Theorem~\ref{thm:improved-EP}. We formulate this lemma because for the
generalizations for hypergraphs we will need its bound on the total 
number of subgraphs used in the partition. For the same choice of $k$ it is
$O(n^2/\log^3n)$.

\begin{lemma}\label k
Let $G$ be a graph on $n$ vertices and let $1\le k\le n$. There exists a
$\BP$-partition of $G$ involving less than $2^kn/k$ complete bipartite
subgraphs such that load of every vertex is at most $2^{k-1}+\lceil n/k\rceil$.

Furthermore, there exists a fractional $\BP$-partition of $G$ involving less
than $2^kn/k$ complete bipartite subgraphs, each with weight $1/2$ or $1$, such
that the load of every vertex is at most $2^{k-2}+\lceil n/k\rceil/2$.
\end{lemma}

\begin{proof}
Let us orient each edge $e$ of $G$ arbitrarily, so now one of its
vertices is $h(e)$, the {\em head} of $e$, while the other is the {\em tail}
$t(e)$ of $e$. We write $N^+(v)$ for the set of {\em outneighbors} of the
vertex $v$, i.e., $N^+(v)=\{h(e):e\in E,t(e)=v\}$. Let us
partition the vertex set into classes $H_1,\dots,H_{\lceil n/k\rceil}$ in
such a way that each class has at most $k$ elements. For a nonempty subset
$S\subseteq H_i$ of a class $H_i$ we consider the complete bipartite graph
$G_S$ whose two partite sets are $S$ and $T_S=\{v\in V:N^+(v)\cap H_i=S\}$.
Figure~\ref{graph_part_fig} illustrates an example of such graph.
The graphs $G_S$ are clearly subgraphs of $G$, their number is less than
$2^kn/k$ as claimed and their edge sets partition the edge set
$E$ as $e\in E$ appears in the unique subgraph $G_S$, where $H_i$ is the class
containing $h(e)$ and $S=N^+(t(e))\cap H_i$. Furthermore a vertex $v\in H_i$
appears in the $2^{|H_i|-1}\le2^{k-1}$ sets $S\subseteq H_i$ and further it
also appears in the partite set $T_S$ of $G_S$ for $S=N^+(v)\cap
H_j\ne\emptyset$, another at most $\lceil n/k\rceil$ graphs. This proves the
first claim of the theorem.

For the second claim we ignore the orientation and work with the full
neighborhood $N(v)=\{w\in V:\{v,w\}\in E\}$ of a vertex. We still use the
same partition of the vertex set. For a nonempty set $S\subseteq H_i$ we
define $G'_S$ to be the complete bipartite graph with partite sets $S$ and
$T'_S=\{v\in V:N(v)\cap H_i=S\}$. It is clear that these graphs are subgraphs of
$G$, their number is the same as the number of the graphs $G_S$, and every
edge $\{v,w\}\in E$ appears in exactly two of these graphs, namely if $v\in
H_i$ and $w\in H_j$, then $\{v,w\}$ appears in $G'_{N(v)\cap H_j}$ and in
$G'_{N(w)\cap H_i}$. Thus, assigning the weight $1/2$ to each of these graphs
we obtain a fractional $\BP$-partition of $G$. The weight $1$ will only show
up if two of the complete bipartite graphs are the same with the role of their
partite sets reversed. As before, a vertex $v\in H_i$ appears in at most
$2^{k-1}$ sets $S\subseteq H_i$ and at most $\lceil n/k\rceil$ sets $T'_S$,
where $S=N(v)\cap H_j\ne\emptyset$ for some $j$.
\end{proof}

\begin{figure}[h!t]
\begin{center}
\begin{tikzpicture}
\foreach\x in {1,2,...,6}{
 \foreach\y in {1,2,...,9}{\draw[fill](\x,0.3*\y) circle (2pt); }
 \draw(\x,1.5) ellipse(0.3 and 1.6);
}
\draw(2.83,1.35) rectangle +(0.335,0.9);
\foreach\y in{5,6,7}{
  \draw(3,0.3*\y)--(5,0.3*8);
  \draw(3,0.3*\y)--(2,2.7);
  \draw(3,0.3*\y) to[out=-145, in=145] (3,0.3);
}
\end{tikzpicture}
\caption{A $K_{3,3}$ subgraph with a 3-element $S$ from the third class}\label{graph_part_fig}\end{center}
\end{figure}
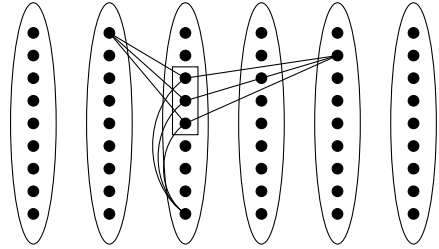

\section{Uniform hypergraphs}\label{sec:hypergraph}

In this section we prove Theorem~\ref{thm:d-uniform}. We note that the
out of the $d$ partite sets of the $d$-cuphs used in the $\BP$-partition or
fractional $\BP$-partition constructed in the proof at least $d-2$ are
singletons. The weight of any $d$-cuph in the fractional $\BP$-partition is a
multiple of $1/(d^2-d)$.

\begin{proof}[Proof of Theorem~\ref{thm:d-uniform}:] Let $\HH=(V,E)$.
We start with something similar to orientation: we partition
every edge $e\in E$ into two sets $A(e)$ and $B(e)$ with $|A(e)|=d-2$ and
$|B(e)|=2$. For $A\subseteq V$, $|A|=d-2$ we define a subset $E_A=\{e\in
E:A(e)=A\}$ of $E$ and a graph $G_A=(V,\{B(e):e\in E_A\})$. See Figure
\ref{hypergraph_fig} for an example in a 6-uniform hypergraph. Clearly,
the sets $E_A$ partition $E$. We apply the claim on $\BP$-partition in
Lemma~\ref k separately to each of the graphs $G_A$ using
$k=\lceil\log n-2\log\log n\rceil$. This yields a partition of the edge set of
$G_A$ into the edge sets of the subgraphs $G_{A,i}$. As calculated before the
statement of the lemma, for every $A$ we have $O(n^2/\log^3n)$ graphs
$G_{A,i}$ and the load of any vertex is at most $(1+o(1))n/\log n$.

\begin{figure}[h!t]
\begin{center}
\begin{tikzpicture}[scale=0.8]
\foreach\x in {1,2,3,4}{\draw (0.5+\x,4) circle (3pt); }
\foreach\x in {1,2,3}{\draw[fill](1+\x,2) circle(3pt); }
\draw(2,2)--(4,2);
\draw(1.2,0) circle (3pt) (2.2,0) circle (3pt);
\draw[dotted] (2.4,0)--(4.4,0);
\draw(4.6,0) circle (3pt);
\draw (0.5,4.3)--(5.5,4.3)--(2.3,0.8)--cycle;
\draw (0.7,4.6)--(6.5,4.6)--(3.5,0.8)--cycle;
\end{tikzpicture}
\caption{A 6-uniform hypergraph with $G_A=K_{2,1}$}\label{hypergraph_fig}\end{center}
\end{figure}
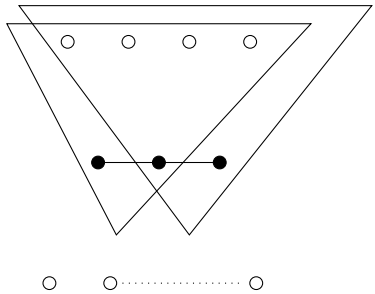

Let $S_{A,i}$ and
$T_{A,i}$ be the partite sets of $G_{A,i}$ and let us define $\mathcal
H_{A,i}$ to be the $d$-cuph with partite sets $S_{A,i}$, $T_{A,i}$ and the
$d-2$ singleton sets contained in $A$. For a fixed $A$ the edge sets of the
$d$-cuphs $\mathcal H_{A,i}$ partition $E_A$. Thus, all the hypergraphs
$\mathcal H_{A,i}$ give a $\BP$-partition of $\HH$.

Let us now fix a vertex $v\in V$ and estimate its load. The vertex $v$ appears
in all $\mathcal H_{A,i}$ with $v\in A$: this contributes $O(n^{d-3}\cdot
n^2/\log^3n)$ to the load. We further have $n-1\choose d-2$ sets $A$ that do
not contain $v$ and at most $(1+o(1))n/\log n$ graphs $G_{A,i}$ for each such
set $A$ in which $v$ appears making $v$ also appear as a vertex of $\mathcal
H_{A,i}$. This brings the total load on $v$ up to at most
$(1/(d-2)!+o(1))n^{d-1}/\log n$ proving the first statement of the theorem.

To obtain a good fractional $\BP$-partition of $\HH$ we
disregard the partition $e=A(e)\cup B(e)$ created earlier. For $A\subseteq V$
with $|A|=d-2$ we define $E'_A=\{e\in E:A\subseteq e\}$ and the graph
$G'_A=(V,\{e\setminus A:e\in E'_A\})$. Every edge $e\in E$ appears
in exactly $d\choose2$ of the sets $E'_A$. Now we apply the claim on
fractional $\BP$-covering in Lemma~\ref k for each of the graphs $G'_A$ with
the same choice of $k$ as before. We obtain a fractional $\BP$-partition of
$G'_A$ with the complete bipartite graphs $G'_{A,i}$ with weight $w_{A,i}$
such that the total number of these graphs is $O(n^2/\log^3n)$ and the maximum
load does not exceed $(1/2+o(1))n^2/\log n$. We define
the $d$-cuphs $\mathcal H'_{A,i}$ as before: its partite sets are the partite
sets of $G'_{A,i}$ plus the singleton sets contained in $A$. We set the weight
of $\mathcal H'_{A,i}$ to be $w_{A,i}/{d\choose2}$. This gives a fractional
$\BP$-partition of $\HH$. Calculating the maximum load of this fractional
$\BP$-partition as before finishes the proof of the theorem.
\end{proof}

\section{Very dense graphs}\label{sec:dense}
In this section we prove Theorem~\ref{uj}. S.\ Jukna in \cite[Theorem
1]{jukna} proves similar results for bipartite graphs with the difference that
he bounds the total number of covering bipartite graphs rather than the
maximum load.
The first part of Theorem~\ref{uj} will follow from the following lemma, 
which will also be used to bootstrap the proof for the $d$-uniform case. 

\begin{lemma}\label{lemma:dense}
Let $G$ be a graph on $n$ vertices such that every vertex has degree 
at least $n-m$. For each $0<p<1$ we have
\begin{equation}\label{eq:lemmadense}
    \vpf(G) \le \frac12\big(p^{-1} + (1-p)^{-m}\big).
\end{equation}
Moreover, the total weight of all graphs taking part in the fractional
partition is $1/(2p(1-p)^m)$.
\end{lemma}

\begin{proof}
The proof uses some ideas from \cite[Lemma 3.7]{beimel-etal}. Let $V$ be the
vertex set of $G$. Choose the random complete bipartite subgraph $H_0$ 
of $G$ as follows. The partite
sets of $H_0$ are $A$ and $B_0$. Construct $A$ by adding each vertex $v$ of
$G$ to the set independently with probability $p$. Let $B_0$ consists of all
vertices in $V\setminus A$ which are connected to all elements in $A$.
Clearly, $H_0$ is a subgraph of $G$.
Now let $H$ be a random subgraph of $H_0$, a complete bipartite graph with
partite sets $A$ and $B$, where $B$ is a subset of $B_0$ that contains the
vertex $v\in B_0$ with probability $(1-p)^{d_v-n+m}$ independently from each
other, where $d_v$ is the degree of $v$ in $G$. 
Note that we assumed that $d_v\ge n-m$, so this selection makes sense. 
This two-step process produces a random complete bipartite subgraph 
$H$ of $G$.

Let us fix a vertex $v\in V$. We have $v\in B_0$ if and only if the
vertices not adjacent to $v$ (including $v$ itself) are not in $A$, thus
$\Prob(v\in B_0)=(1-p)^{n-d_v}$. By our construction, $v\in B$ implies $v\in
B_0$ and we have $\Prob(v\in B|v\in B_0)=(1-p)^{d_v-n+m}$. Thus, overall, we
have $\Prob(v\in B)=(1-p)^m$. Note also, that $v\in B$ is independent of $u\in
A$ for all vertices $u$ adjacent to $v$. Thus, for every $uv$ edge of $G$ we
have $\Prob(u\in A,v\in B)=p(1-p)^m$. The same $uv$ edge is also contained in
$H$ if $u\in B$ and $v\in A$, thus the total probability of an edge $uv$ of
$G$ to be covered by $H$ is exactly $2p(1-p)^m$.

Let us associate with each complete bipartite subgraph $H^*$ of $G$ the weight
$\Prob(H=H^*)/(2p(1-p)^m)$. The calculation above verifies that this is a
fractional $\BP$-partition. The load of every vertex $v$ is exactly
$$
 \frac{\Prob(v\in A\cup B)}{2p(1-p)^m} 
    =\frac{p+(1-p)^m}{2p(1-p)^m}
    =\frac12\big(p^{-1}+(1-p)^{-m}\big),
$$
as was required. The total weight of the complete bipartite subgraphs is
$1/(2p(1-p)^m)$ as the sum of probabilities is exactly 1.
\end{proof}

\begin{proof}[of Theorem~\ref{uj}.]
We choose $p^{-1} = m\log e/ (\log m - 2\log\log m)$ where $e$ is the 
base of the natural logarithm. With this choice the right hand side of 
(\ref{eq:lemmadense}) is
$$
  \big(\frac{\log e}2+o(1)\big)\frac m{\log m} <
  \big(0.722+o(1)\big)\frac m{\log m},   
$$
which proves the first claim of Theorem~\ref{uj}.
Note that the total weight of all graphs taking part of the fractional
partition is $1/(2p(1-p)^m)=O(m^2/\log^2m)$.

To turn this fractional partition result on graphs to one on hypergraphs we do
the same as in the proof of Theorem~\ref{thm:d-uniform}. Let $\HH=(V,E)$ be a
$d$-uniform hypergraph. For a $(d-2)$-set $A$ of vertices consider the edge
set $E_A=\{e\in E\mid A\subseteq e\}$ and the graph $G_A=(V,\{e\setminus A\mid
e\in E_A\})$. Ignoring the isolated vertices in $G$ we get a graph on less
than $n$ vertices and with minimum degree at least $n-m$. Applying our result
above we obtain a fractional $\BP$-partition of $G_A$ with complete bipartite
graphs $G_{A,i}$ with weight $w_{A,i}$. The claimed fractional $\BP$-partition
of $\HH$ is formed by the hypergraphs $\HH_{A,i}$ with weights
$w_{A,i}/{d\choose2}$, where the partite sets of $\HH_{A,i}$ are those of
$G_{A,i}$ and the singleton sets contained in $A$. Calculation of the load of
this partition finishes the proof of the theorem.
\end{proof}

\section{Lower bound for the load of fractional covers}\label{sec:this-is-best}
To show Theorem~\ref{thm:upper-bound} and see that the fractional results in
Theorems~\ref{thm:improved-EP} and \ref{thm:d-uniform} are optimal within a
factor of less than $2$ we turn to random hypergraphs. Let $\mathcal H^d(n,p)$
denote the random $d$-uniform hypergraph on $n$
vertices in which each $d$-subset of the vertices is an edge with probability
$p$ and these events are independent. For a $d$-uniform hypergraph $\mathcal
H=(V,E)$ we write $\rho(\mathcal H)=|E|/|V|$ and call it the {\em density} of
$\mathcal H$. Note that $d\rho(\mathcal H)$ is the {\em average  degree} in
$\mathcal H$.

\begin{lemma}\label{density}
Let $d\ge2$ be an integer and $0<p<1$. With probability tending to $1$
as $n$ goes to infinity the maximum density of a $(d,k)$-cuph subhypergraph
of $\mathcal H^d(n,p)$ (for any $k$) is at most $-\log n/\log p$.
\end{lemma}

\begin{proof} We use the first moment method. Let $\mathcal H=(V,E)$ a
$(d,k)$-cuph with $\rho(\mathcal H)>-\log n/\log p$. We can get rid of the
empty partite sets and assume that all $k$ partite sets of $\mathcal H$ is
non-empty. For a fixed size $s=|V|$ and $k$ we
have $n\choose s$ possibilities to choose $V$ as a subset of the fixed vertex
set of $\mathcal H^d(n,p)$ and less than $s^k/k!$ ways split it into the
partite sets that determine $\mathcal H$. For a fixed $\mathcal H$ the chance
that it is a subhypergraph of $\mathcal H^d(n,p)$ is $p^{|E|}=p^{\rho(\mathcal
  H)s}<n^{-s}$. Thus the probability of the existence (in fact the expected
number) of suitably dense
$(d,k)$-cuph subhypergraph (for any $k$) can be estimated as less than
$$\sum_{s,k}{n\choose s}\frac{s^k}{k!}n^{-s}<\sum_{s,k}\frac{s^k}{s!k!}.$$
Note that this estimate is a convergent sum. If summed for all $s,k\ge0$ we
obtain $e^e$, where $e$ is the base of the natural logarithm. However
$d$-uniform hypergraphs of any given fixed size $s$ have a bounded density, so
as $n$ increases and our threshold $-\log n/\log p$ passes this density we can
ignore small values of $s$. This yields a sum that tends to zero as claimed.
\end{proof}

\begin{thm}\label{random_thm}
Let $d\ge 2$ be an integer, $0<p<1$ and $\epsilon>0$. With probability tending
to $1$ as $n$ goes to infinity we have $\vcf_\MP(\mathcal H^d(n,p))\ge(-p\log p/d!-\epsilon)n^{d-1}/\log n$.
\end{thm}

\begin{proof} Let $\mathcal H^d(n,p)=(V,E)$ and let the hypergraphs $\mathcal
H_i=(V_i,E_i)$ with the weights $w_i$ form a fractional $\MP$-cover of
$\mathcal H^d(n,p)$. By Lemma~\ref{density} we may
assume that $\rho(\mathcal H_i)=|E_i|/|V_i|\le-\log n/\log p$ for all
$i$. This means $w_i|E_i|\le-w_i\log n|V_i|/\log p$. Summing these
inequalities we get
$$|E|\le\sum_{e\in E}\sum_{i:e\in E_i}w_i=\sum_iw_i|E_i|\le-\frac{\log n}{\log
  p}\sum_iw_i|V_i|=-\frac{\log n}{\log p}\sum_{v\in V}l_v,$$
where $l_v$ is the load of the vertex $v$. Thus, for the maximum load $l$ we
have
$$l\ge-\frac{\log p}{\log n}\rho(\mathcal H^d(n,p)).$$
Here the expectation of $\rho(\mathcal H^d(n,p))$ is $\hbox{Exp}
[|E|]/n=p{n\choose d}/n=(p/d!+o(1))n^{d-1}$. Note further that the
distribution of $|E|$ is binomial, and therefore it is concentrated around its
expectation. That is, with probability tending to $1$ we have
$\rho(\mathcal H^d(n,p))\ge(p/d!+\epsilon/\log p)n^{d-1}$. This gives $l\ge(-p\log
p/d!-\epsilon)n^{d-1}/\log n$. As we proved this bound for the maximal load of
an arbitrary fractional $\MP$-cover of $\mathcal H^d(n,p)$ it also applies to
$\vcf_\MP(\mathcal H^d(n,p))$ (still with probability tending to $1$), and finishes the
proof of the theorem.
\end{proof} 

\begin{proof}[Proof of Theorem~\ref{thm:upper-bound}:]
By Theorem~\ref{random_thm} for any $p$ and $\epsilon>0$ and large enough $n$
there exists a hypergraph $\mathcal H$ on $n$ vertices with $\vcf_\MP(\mathcal
H)\ge(-p\log p/d!-\epsilon)n^{d-1}/\log n$, namely the random graph $\mathcal
H^d(n,p)$ works with high probability. Here $-p\log p$ is maximized for
$p=1/e$, where $e$ is the base of the natural logarithm. This choice for $p$
proves the theorem.
\end{proof}

\section*{Acknowledgment}
This research has been partially supported by the Lend\"ulet program of the
Hungarian Academy of Sciences. The first author also acknowledges 
the support from the grant TAMOP-4.2.2.C-11/1/KONV-2012-0001. The second
author was supported by the OTKA grant PD100712. The last author
also acknowledges the support of the grant OTKA NN-102029 and the NSERC
Discovery grant.

\end{document}